\documentclass{article}

\usepackage{graphicx}
\usepackage{indentfirst}
\usepackage{amsmath,amsfonts,amsthm,amssymb}
\usepackage{mathrsfs}
\usepackage{amscd}

\def\co{\colon\thinspace}
\DeclareMathAlphabet{\mathsfsl}{OT1}{cmss}{m}{sl}

\newtheorem{thm}{Theorem}[section]
\newtheorem{lem}[thm]{Lemma}

\newtheorem{prop}[thm]{Proposition}

\theoremstyle{definition}
\newtheorem{defn}[thm]{Definition}

\newtheorem{correction}[thm]{Correction}

\begin{document}

\title{Dehn surgeries on knots in product manifolds}

\author{{\Large Yi NI}\\{\normalsize Department of Mathematics, Caltech, MC 253-37}\\
{\normalsize 1200 E California Blvd, Pasadena, CA
91125}\\{\small\it Emai\/l\/:\quad\rm yni@caltech.edu}}

\date{}
\maketitle

\begin{abstract}
We show that if a surgery on a knot in a product sutured manifold
yields the same product sutured manifold, then this knot is a
$0$-- or $1$--crossing knot. The proof uses techniques from
sutured manifold theory.
\end{abstract}

{\it\hfill Dedicated to the memory of Professor Andrew Lange}

\section{Introduction}

An interesting problem on Dehn surgery is: when does a surgery on
a knot yield a manifold homeomorphic to the original ambient
manifold? The most famous result in this direction is the Knot
Complement Theorem proved by Gordon and Luecke \cite{GL}: when the
ambient manifold is $S^3$, only the unknot admits surgeries which
yield $S^3$.

In this paper, we are going to study this problem for knots in
surfaces times an interval. Our main result is as follows.

\begin{thm}\label{thm:SurgProd}
Suppose $F$ is a compact surface, $K\subset F\times I$ is a knot.
Suppose $\alpha$ is a nontrivial slope on $K$, and $N(\alpha)$ is
the manifold obtained from $F\times I$ via the $\alpha$--surgery
on $K$. If the pair $(N(\alpha),(\partial F)\times I)$ is
homeomorphic to the pair $(F\times I,(\partial F)\times I)$, then
one can isotope $K$ such that its image on $F$ under the natural
projection $$p\co F\times I\to F$$ has either no crossing or
exactly one crossing.

The slope $\alpha$ can be determined as follows. Let $\lambda_b$
be the ``blackboard'' frame of $K$ associated with the previous
projection. Namely, $\lambda_b$ is the frame specified by the
surface $F$. When the projection has no crossing, $\alpha=\frac1n$
for some integer $n$ with respect to $\lambda_b$; when the minimal projection has exactly one
crossing, $\alpha=\lambda_b$.
\end{thm}

It is easy to see the surgeries in the statement of
Theorem~\ref{thm:SurgProd} do not change the homeomorphism type of
the pair $(F\times I,(\partial F)\times I)$. In fact, when $K$ is
a $0$--crossing knot, it is clear that the $\frac1n$--surgery preserves
the homeomorphism type of the pair. When $K$ is a
$1$--crossing knot, we can add a one-handle to $F\times\frac12$
near the crossing to get a Heegaard surface $F'$ for $F\times I$.
$K$ can be embedded into $F'$ as in Figure~\ref{fig:onehandle}.
$F'$ splits $F\times I$ into two parts $U_0,U_1$, where $U_0$ is
$F\times[0,\frac12]$ with a one-handle added to $F\times\frac12$,
and $U_1$ is $F\times[\frac12,1]$ with a one-handle added to
$-F\times\frac12$. The embedding of $K$ can be chosen such that
$K$ goes through each of the two one-handles exactly once. Now the
blackboard frame $\lambda_b$ is the frame specified by $F'$, and
the $\lambda_b$--surgery on $F'$ cancels each one-handle with a
two-handle. Hence the new pair is still homeomorphic to $(F\times
I,(\partial F)\times I)$.

\begin{figure}
\begin{picture}(340,110)
\put(70,0){\scalebox{0.5}{\includegraphics*[10pt,435pt][410pt,
655pt]{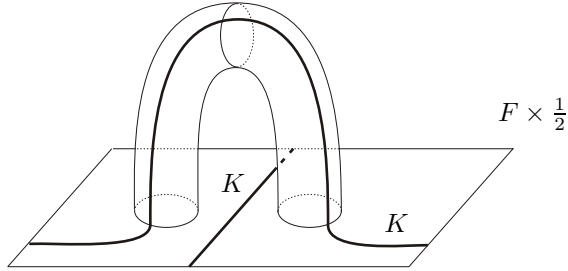}}}

\put(260,60){$F\times\frac12$}

\put(155,33){$K$}

\put(217,18){$K$}

\end{picture}
\caption{A local picture of the crossing}\label{fig:onehandle}
\end{figure}

\begin{defn}
Notations are as in the previous theorem. Fix a product structure
on $(\partial F)\times I$. Up to an isotopy relative to $(\partial
F)\times I$, this product structure uniquely extends to a product
structure $\mathcal P$ on $F\times I$ and a product structure
$\mathcal P_{\alpha}$ on $N(\alpha)$. (This fact can be proved
using Alexander's trick.) Identify $F$ with $F\times1$. Let
$i,i_{\alpha}\co F\times0\to F\times 1$ be the natural identity
maps with respect to $\mathcal P $ and $\mathcal P_{\alpha}$,
respectively.  We call
$$\varphi_{\alpha}=i\circ i_{\alpha}^{-1}\co F\to F$$ {\it the
map induced by the $\alpha$--surgery}. This map $\varphi_{\alpha}$
fixes $\partial F$ pointwise, and is unique up to an isotopy
relative to $\partial F$. Hence $\varphi_{\alpha}$ can be viewed
as an element in the mapping class group $\mathcal{MCG}(F,\partial
F)$.
\end{defn}

The definition of the map $\varphi_{\alpha}$ is justified by the
following lemma.

\begin{lem}\label{lem:Monodromy}
Let $Y(\alpha)$ be the manifold obtained from $F\times S^1$ by
$\alpha$--surgery on $K$. Then $Y(\alpha)$ can be obtained from
$F\times I$ by identifying $(x,0)$ with $(\varphi_{\alpha}(x),1)$
for any $x\in G$.
\end{lem}
\begin{proof}
The manifold $F\times S^1$ is obtained from $F\times I$ by
identifying $y$ with $i(y)$ for each $y\in F\times0$. Let
$y=(x,0)$ with respect to the product structure $\mathcal
P_{\alpha}$ on $N(\alpha)$, then $i_{\alpha}(y)=(x,1)$ with
respect to $\mathcal P_{\alpha}$. We then have
$$i(y)=\varphi_{\alpha}(x,1)=(\varphi_{\alpha}(x),1),$$
since we identify $F$ with $F\times 1$ in the above definition.
Hence $(x,0)$ is identified with $(\varphi_{\alpha}(x),1)$ in
$Y(\alpha)$ for each $x\in F$.
\end{proof}

\begin{prop}\label{prop:MapEle}
Notations are as in Theorem~\ref{thm:SurgProd}. When the
projection of $K$ has no crossing and $\alpha=\frac1n$,
$$\varphi_{\alpha}=\tau^n,$$
where $\tau$ is the righthand Dehn twist along $K\subset F$. When the minimal
projection of $K$ has exactly one crossing, let $a,b,c$ be the
simple closed curves obtained by resolving the crossing in two different ways as in Figure~\ref{fig:1Crossing} and let
$\tau_a,\tau_b,\tau_c$ be the righthand Dehn twists along $a,b,c$.
Then $$\varphi_{\alpha}=\tau_a^2\tau_b^2\tau_c^{-1}$$ when the
crossing is positive, and
$$\varphi_{\alpha}=\tau_a^{-2}\tau_b^{-2}\tau_c$$
when the crossing is negative.
\end{prop}

\begin{figure}
\begin{picture}(340,130)
\put(50,0){\scalebox{0.50}{\includegraphics*[50pt,190pt][530pt,
450pt]{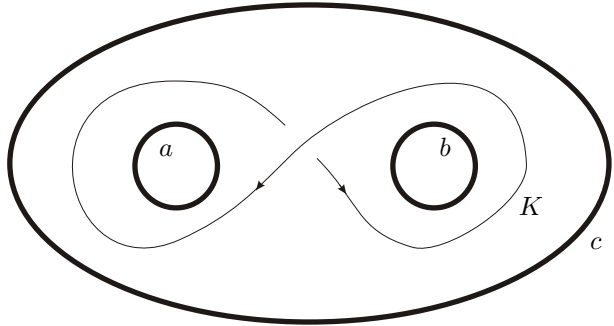}}}

\put(112,70){$a$}

\put(218,70){$b$}

\put(275,35){$c$}

\put(248,48){$K$}

\end{picture}
\caption{A $1$--crossing knot}\label{fig:1Crossing}
\end{figure}

This paper can be compared with Ni \cite{Ni2}. In fact,
Theorem~1.4 in \cite{Ni2} can be restated in a form similar to
Theorem~\ref{thm:SurgProd}.

\begin{thm}\label{thm:NormRed}
Suppose $F$ is a compact surface, $K\subset F\times I$ is a knot
and $\alpha$ is a slope on $K$. Let $N(\alpha)$ be the manifold
obtained by the $\alpha$--surgery on $K$. If $F\times \{0\}$ is
not Thurston norm minimizing in $H_2(N(\alpha),(\partial F)\times
I)$, then there is an ambient isotopy of $F\times I$ which takes
$K$ to a curve in $F\times\{\frac12\}$. Moreover, $\alpha$ is the
frame on $K$ specified by $F\times\{\frac12\}$.
\end{thm}

The proof of Theorem~\ref{thm:NormRed} uses Gabai's sutured
manifold theory \cite{G1,G2,G3} and an argument due to Ghiggini
\cite{Gh}. Using a different method, Scharlemann and Thompson
\cite{SchTh2} get the same conclusion of Theorem~\ref{thm:NormRed}
under the assumption that $F\times\{0\}$ is compressible in
$N(\alpha)$.

This paper is organized as follows. In Section~2, we give some
preliminaries on sutured manifold theory and foliations, as well
as a characterization of one-crossing knot projections. In
Section~3, we study some warm-up cases. In Section~4, we use the
argument in the proof of Theorem~\ref{thm:NormRed} to reduce our
problem to the case where $F$ is a pair of pants. In Section~5, we
study this case by analyzing the map induced by surgery and using
a variant of the argument in Ni \cite{Ni2}.

\

\noindent{\bf Acknowledgements.}\quad We are grateful to Danny
Calegari for helpful discussions on mapping class groups. The
author was partially supported by an AIM Five-Year Fellowship and
NSF grant number DMS-0805807.

\section{Preliminaries}

In this section, we are going to review the sutured manifold
theory introduced by Gabai in \cite{G1}. We also state a
uniqueness result for the Euler classes of taut foliations of
fibred manifolds. In addition, we define ``double primitive''
knots in $F\times I$ and show that they are exactly the knots with
a projection consisting of only one crossing.

\subsection{Sutured manifold decompositions}

\begin{defn}
A {\it sutured manifold} $(M,\gamma)$ is a compact oriented
3--manifold $M$ together with a set $\gamma\subset \partial M$ of
pairwise disjoint annuli $A(\gamma)$ and tori $T(\gamma)$. The
core of each component of $A(\gamma)$ is a {\it suture}, and the
set of sutures is denoted by $s(\gamma)$.

Every component of $R(\gamma)=\partial M-\mathrm{int}(\gamma)$ is
oriented. Define $R_+(\gamma)$ (or $R_-(\gamma)$) to be the union
of those components of $R(\gamma)$ whose normal vectors point out
of (or into) $M$. The orientations on $R(\gamma)$ must be coherent
with respect to $s(\gamma)$, hence every component of $A(\gamma)$
lies between a component of $R_+(\gamma)$ and a component of
$R_-(\gamma)$.
\end{defn}

\begin{defn}
Let $S$ be a compact oriented surface with connected components
$S_1,\dots,S_n$. We define
$$x(S)=\sum_i\max\{0,-\chi(S_i)\}.$$
Let $M$ be a compact oriented 3--manifold, $A$ be a compact
codimension--0 submanifold of $\partial M$. Let $h\in H_2(M,A)$.
The {\it Thurston norm} $x(h)$ of $h$ is defined to be the minimal
value of $x(S)$, where $S$ runs over all the properly embedded
surfaces in $M$ with $\partial S\subset A$ and $[S]=h$.
\end{defn}

\begin{defn}
Let $(M,\gamma)$ be a sutured manifold, and $S$ a properly
embedded surface in M, such that no component of $\partial S$
bounds a disk in $R(\gamma)$ and no component of $S$ is a disk
with boundary in $R(\gamma)$. Suppose that for every component
$\lambda$ of $S\cap\gamma$, one of 1)--3) holds:

1) $\lambda$ is a properly embedded non-separating arc in
$\gamma$.

2) $\lambda$ is a simple closed curve in an annular component $A$
of $\gamma$ in the same homology class as $A\cap s(\gamma)$.

3) $\lambda$ is a homotopically nontrivial curve in a toral
component $T$ of $\gamma$, and if $\delta$ is another component of
$T\cap S$, then $\lambda$ and $\delta$ represent the same homology
class in $H_1(T)$.

Then $S$ is called a {\it decomposing surface}, and $S$ defines a
{\it sutured manifold decomposition}
$$(M,\gamma)\stackrel{S}{\rightsquigarrow}(M',\gamma'),$$
where $M'=M-\mathrm{int}(\mathrm{Nd}(S))$ and
\begin{eqnarray*}
\gamma'\;\;&=&(\gamma\cap M')\cup \mathrm{Nd}(S'_+\cap
R_-(\gamma))\cup
\mathrm{Nd}(S'_-\cap R_+(\gamma)),\\
R_+(\gamma')&=&((R_+(\gamma)\cap M')\cup S'_+)-\mathrm{int}(\gamma'),\\
R_-(\gamma')&=&((R_-(\gamma)\cap M')\cup
S'_-)-\mathrm{int}(\gamma'),
\end{eqnarray*}
where $S'_+$ ($S'_-$) is that component of
$\partial\mathrm{Nd}(S)\cap M'$ whose normal vector points out of
(into) $M'$.
\end{defn}

\begin{defn}
A sutured manifold $(M,\gamma)$ is {\it taut}, if $M$ is
irreducible and $R(\gamma)$ is  Thurston norm
minimizing in $H_2(M,\gamma)$.

Suppose $S$ is a decomposing surface in $(M,\gamma)$, $S$
decomposes $(M,\gamma)$ to $(M',\gamma')$. $S$ is {\it taut} if
$(M',\gamma')$ is taut.
\end{defn}

\begin{defn}
Suppose $$(M,\gamma)\stackrel{S}{\rightsquigarrow}(M',\gamma')$$
is a taut decomposition, by \cite{G1} we can extend this
decomposition to a sutured manifold hierarchy of $(M,\gamma)$,
from which we can construct a taut foliation $\mathscr F$ of $M$,
such that $R(\gamma)$ consists of compact leaves of $\mathscr F$.
We then call $\mathscr F$ a {\it foliation  induced by $S$}.
Moreover, when $R_+(\gamma)$ is homeomorphic to $R_-(\gamma)$,
from $M$ we can obtain a manifold $Y$ with boundary consisting of
tori by gluing $R_+(\gamma)$ to $R_-(\gamma)$ via a homeomorphism.
$\mathscr F$ then becomes a taut foliation $\mathscr F_1$ of $Y$.
We also say that $\mathscr F_1$ is a {\it foliation  induced by
$S$}.
\end{defn}

\begin{defn}
A decomposing surface is called a {\it product disk}, if it is a
disk which intersects $s(\gamma)$ in exactly two points. A
decomposing surface is called a {\it product annulus}, if it is an
annulus with one boundary component in $R_+(\gamma)$, and the
other boundary component in $R_-(\gamma)$.
\end{defn}

We recall the main result in Gabai \cite{G2}, which has been
intensively used in Ni \cite{Ni2}. Note that the result is not
stated in its original form, but it is contained in the argument
in \cite{G2}. See also \cite[Theorem~2.8]{Ni2} for a sketch of the
proof.

\begin{defn}
An {\it I-cobordism} between closed connected surfaces $T_0$ and
$T_1$ is a compact $3$--manifold $V$ such that $\partial V=T_0\cup
T_1$ and for $i=0,1$ the induced maps $j_i\co H _1(T_i)\to H_1(V)$
are injective.
\end{defn}

\begin{defn}
Suppose $M$ is a 3-manifold, $T$ is a toral component of $\partial
M$. If all tori in $M$ which are I-cobordant to $T$ in $M$ must be
parallel to $T$, then we say $M$ is {\it $T$--atoroidal}.
\end{defn}

\begin{thm}[Gabai]\label{thm:Gabai}
Let $(M,\gamma)$ be a taut sutured $3$--manifold. $T$ is a toral
component of $\gamma$, $S$ is a decomposing surface such that
$S\cap T=\emptyset$, and the decomposition
$$(M,\gamma)\stackrel{S}{\rightsquigarrow}(M_1,\gamma_1)$$
is taut. Suppose $M$ is $T$--atoroidal, then for any slope
$\alpha$ on $T$ except at most one slope, the decomposition after
Dehn filling
$$(M(\alpha),\gamma\backslash T)\stackrel{S}{\rightsquigarrow}(M_1(\alpha),\gamma_1\backslash T)$$
is taut.\qed
\end{thm}

A special case of the above theorem is the case $\gamma=\partial
M$, which is the original form in \cite{G2}.

\subsection{Euler classes of foliations}

We will need the Euler classes of foliations.

\begin{defn}
Suppose $Y$ is a compact $3$--manifold with $\partial Y$
consisting of tori. $\mathscr P$ is an oriented plane field transverse to $\partial Y$. Let $T(\partial Y)$ be the tangent
plane field of $\partial Y$. The line field $\mathscr P\cap
T(\partial Y)$ has a natural orientation induced by the
orientations of $\mathscr P$ and $T(\partial Y)$, thus it has a
nowhere vanishing section $v\subset \mathscr P|_{\partial Y}$.
Then one can define the {\it relative Euler class}
$$e(\mathscr P)\in H^2(Y,\partial Y)$$
of $\mathscr P$ to be the obstruction to extending $v$ to a
nowhere vanishing section of $\mathscr P$. When $\mathscr F$ is a
foliation of $Y$ that is transverse to $\partial Y$, let $T\mathscr F$
be the tangent plane field of $\mathscr F$ and let
$e(\mathscr F)=e(T\mathscr F)$.
\end{defn}

\begin{defn}
Suppose $C$ is a properly embedded curve in a compact surface $F$.
We say $C$ is {\it efficient} in $F$ if
$$|C\cap\delta|=|[C]\cdot[\delta]|,\quad\text{for each boundary component $\delta$ of $F$.}$$
Suppose $S$ is a properly embedding surface in compact
$3$--manifold $Y$ with boundary consisting of tori. We say $S$ is
{\it efficient} in $Y$ if $S\cap T$ consists of coherently
oriented
 parallel essential curves for each boundary component $T$ of $Y$.
\end{defn}

\begin{prop}\label{prop:EuEqual}
Suppose $Y$ is a compact $3$--manifold that fibres over $S^1$. Let
$G$ be a fibre of the fibration $\mathscr E$. Suppose $\mathscr F$
is a taut foliation of $Y$ which is transverse to $\partial Y$
such that $G$ is a leaf of $\mathscr F$. Then
$$e(\mathscr F)=e(\mathscr E)\in H^2(Y,\partial Y)/\mathrm{Tors}.$$
\end{prop}
\begin{proof}
This result follows easily from the fact that the Floer homology
of a fibred manifold is ``monic''. Using this approach, one can
even prove that the two Euler classes are equal in $H^2(Y,\partial
Y)$. Here we will present a more geometric proof.

In order to prove the desired result, we only need to show that
\begin{equation}\label{eq:CheckEv}
\langle e(\mathscr F),h\rangle=\langle e(\mathscr E),h\rangle
\end{equation}
for any $h\in H_2(Y,\partial Y)$. When $h=[G]$, we have
\begin{equation}\label{eq:CheckG}
\langle e(\mathscr F),[G]\rangle=\langle e(\mathscr
E),[G]\rangle=\chi(G).
\end{equation}

In general, suppose $\overline U\subset Y$ is a proper surface
representing $h$ such that $\overline U\pitchfork G$. We can
choose the representative $\overline U$ such that $\overline U$ is
efficient in $Y$. Then $\overline U\cap G$ can also be made
efficient in $G$. Cutting $Y$ open along $G$, we get $G\times I$.
Let $U\subset G\times I$ be the proper surface obtained by cutting
$\overline U$ open along $C=\overline U\cap G$. Let
$C_0,C_1\subset G$ be proper oriented curves such that
$$-C_0\times 0=(\partial U)\cap (G\times 0),\quad C_1\times 1=(\partial U)\cap
(G\times 1).$$

Since $C_0$ and $C_1$ are homologous efficient curves in $G$
relative to $\partial G$, as in the proof of \cite[Lemma~0.6]{G2},
we can find compact subsurfaces $V_1,V_2,\dots,V_n$ and efficient
curves
$$C_0=\gamma_0,\gamma_1,\dots,\gamma_{n-1},\gamma_n=C_1$$ in $G$,
such that $$\overline{\partial V_i\backslash(\partial
G)}=\gamma_i\cup(-\gamma_{i-1}).$$ Let $W_i=\overline{G\backslash
V_i}$. Perturbing the surface
\begin{equation}\label{eq:ConstructV}
\bigcup_{i=1}^n\left((-\gamma_{i-1}\times[\frac
{i-1}n,\frac{i}n])\cup(V_i\times\frac in)\right)
\end{equation} slightly, we
get a proper surface $V\subset G\times I$, such that
$$(\partial V)\cap (G\times 0)=-C_0\times 0,\quad (\partial V)\cap
(G\times 1)=C_1\times 1.$$ Similarly, perturbing the surface
\begin{equation}\label{eq:ConstructW}
\bigcup_{i=1}^n\left((\gamma_{i-1}\times[\frac
{i-1}n,\frac{i}n])\cup(W_i\times\frac in)\right)
\end{equation} slightly, we
have a proper surface $W\subset G\times I$, such that
$$(\partial W)\cap (G\times 0)=C_0\times 0,\quad (\partial W)\cap
(G\times 1)=-C_1\times 1.$$

Let $\overline V\subset Y$ be the proper surface obtained from $V$
by identifying $C_0\times0$ and $C_1\times1$ with $C\subset
G\subset Y$. Similarly, define $\overline W\subset Y$. Note that
$$[\overline V]-[\overline U]=[V\cup(-U)]\in H_2(Y,\partial Y).$$
Perturbing $V\cup(-U)$ slightly, we get a properly immersed
surface in $Y$ which is disjoint from the fibre $G$. So
$[V\cup(-U)]=m[G]$ for some integer $m$. Using (\ref{eq:CheckG}),
in order to check (\ref{eq:CheckEv}) for $h=[\overline U]$, we only need to
check it for $h=[\overline V]$.

Since $\mathscr F$ is taut, by Thurston \cite[Corollary~1]{T} we
have
\begin{eqnarray*}
\chi(\overline V)&\le& \langle e(\mathscr F),[\overline V]\rangle,\\
\chi(\overline W)&\le& \langle e(\mathscr F),[\overline W]\rangle.
\end{eqnarray*}
Adding the two inequalities together, we get
\begin{equation}\label{eq:ChiIneq}
\chi(\overline V)+\chi(\overline W)\le\langle e(\mathscr
F),[\overline V]+[\overline W]\rangle.
\end{equation}
By the constructions (\ref{eq:ConstructV}), (\ref{eq:ConstructW}),
the result of doing oriented cut-and-pastes to $\overline V$ and
$\overline W$ is $n$ copies of $G$. So the left hand side of
(\ref{eq:ChiIneq}) is $n\chi(G)$, while the right hand side is
$\langle e(\mathscr F),n[G]\rangle=n\chi(G)$. So the equality
holds. In particular, we should have
$$\chi(\overline V)=\langle e(\mathscr F),[\overline V]\rangle.$$
The same argument shows that
$$\chi(\overline V)=\langle e(\mathscr E),[\overline V]\rangle,$$
so (\ref{eq:CheckEv}) holds for $h=[\overline V]$. Since we have
checked (\ref{eq:CheckEv}) for all elements $h\in H_2(Y,\partial
Y)$, $e(\mathscr F)$ is equal to $e(\mathscr E)$ up to a torsion
element in $H^2(Y,\partial Y)$.
\end{proof}

\subsection{A characterization of one-crossing knot projections}

In this subsection, we will give a characterization of
one-crossing knot projections in terms of double primitive knots. This
fact is not used in the current paper, but it is useful to bare it
in mind.

\begin{defn}
Let $F'\subset F\times I$ be a connected surface of genus
$g(F)+1$, and $\partial F'=(\partial F)\times\frac12$. Suppose
$F'$ is a Heegaard surface. Namely, $F'$ splits $F\times I$ into
two parts $U_0$ and $U_1$, such that $U_0$ is homeomorphic to
$(F\times[0,\frac12])\cup H_1$, and $U_2$ is homeomorphic to
$(F\times[\frac12,1])\cup H_2$, where $H_1$ is a one-handle with
feet on $F\times \frac12$ and $H_2$ is a one-handle with feet on
$-F\times \frac12$. A knot $K\subset F\times I$ is a {\it double
primitive} knot if it is isotopic to a curve on $F'$ which goes
through each of $H_1,H_2$ exactly once.
\end{defn}

\begin{lem}
A knot $K\subset F\times I$ is double primitive if and only if it
has a projection which has only one crossing.
\end{lem}
\begin{proof}
If a knot has a one-crossing projection, then it is double
primitive as shown in the introduction. Now assume $K$ is double
primitive, then $K$ is embedded into a Heegaard surface $F'$ as in
the above definition.

We claim that $F'$ is stabilized. Namely, there is a compressing
disk $D_0\subset U_0$ and a compressing disk $D_1\subset U_1$ such
that $|(\partial D_0)\cap(\partial D_1)|=1$. When $F$ is closed,
this follows from the theorem of Scharlemann and Thompson
\cite{SchTh1} that the Heegaard splittings of $F\times I$ are
standard. When $F$ is not closed, let $R$ be the torus with one
hole, we can glue a copy of $R$ to each component of $\partial F$,
then $F$ becomes a closed surface $G$ and $F'$ becomes a Heegaard
surface $G'$ in $G\times I$. Using Scharlemann and Thompson's
theorem, $G'$ is stabilized, hence there are compressing disks
$D_0$ and $D_1$ in the two compression bodies separated by $G'$,
such that $|(\partial D_0)\cap(\partial D_1)|=1$. Using standard
arguments we can isotope $D_0$ and $D_1$ to be disjoint from the
copies of $R\times I$, so $D_0\subset U_0$, $D_1\subset U_1$, thus
our conclusion follows.

Since $g(F')=g(F)+1$, after compressing $F'$ along $D_0$ we get a
surface homeomorphic to $F$ (and hence parallel to $F\times0$ in
$F\times I$). So $F'$ is obtained from $F\times\frac12$ by adding
a one-handle, and $D_1$ is a disk whose boundary goes through the
one-handle exactly once. Now the local picture of $F'$ looks
exactly like in Figure~\ref{fig:onehandle}. The knot $K$ goes
through the one-handle once and intersects $\partial D_1$ once, so
there is a crossing near $D_1$ and no crossing elsewhere.
\end{proof}

\section{Warm-up cases}\label{sect:WarmUp}

In this section, we are going to prove some easy cases of our main
theorem. When $F$ is a disk or sphere, our result follows from
Gordon and Luecke's Knot Complement Theorem \cite{GL}. When $F$ is
an annulus, we have the following lemma.

\begin{lem}\label{lem:FAnnulus}
Theorem~\ref{thm:SurgProd} is true when $F$ is an
annulus.
\end{lem}
\begin{proof}
Let $\mathcal M$ be the meridian of the solid torus $V=F\times I$,
and let $\mathcal L$ be the frame of $V$ specified by $\partial
F$. By Gabai \cite{GSurg}, if $K$ is nontrivial, then $K$ is a
$0$-- or $1$--bridge braid in $F\times I$.

Capping off one boundary component of $F$ with a disk, we get a
disk $D$.  Let $\lambda$ be the Seifert frame of $K$ in $D\times
I$ and let $\mu$ be the meridian of $K$.

If $K$ is the core of $V$, then the surgery preserves the
homeomorphism type of $(F\times I,(\partial F)\times I)$ if and
only if the slope is $\mu+n\lambda$ for some integer $n$.

From now on we assume the braid index of $K$ is greater than $1$.

If $K$ is a $0$--bridge braid, then $K$ is isotopic to $p\mathcal
L+q\mathcal M$ on $\partial V$ for some $p,q\in\mathbb Z$. Let
$\Lambda$ be the frame on $K$ specified by $\partial V$, then
$\Lambda=pq\mu+\lambda$. A surgery on $K$ yields a solid torus if
and only if the slope $\alpha$ of the surgery satisfies that
$\Delta(\alpha,\Lambda)=1$, namely, when the slope $\alpha$ is
$\mu+n\Lambda$ for some integer $n$. Now $p\alpha=p\mu+pn\Lambda$
is homologous to $\mathcal M+pn(p\mathcal L+q\mathcal M)$ in
$V\backslash K$, so the meridian of the new ambient solid torus
after surgery is $(1+pqn)\mathcal M+p^2n\mathcal L$. Since the
surgery preserves the homeomorphism type of the pair $(F\times
I,(\partial F)\times I)$, we must have $\Delta((1+pqn)\mathcal
M+p^2n\mathcal L,\mathcal L)=1$, thus $1+pqn=\pm1$. Since
$p>1,n\ne0$, we have $(p,q,n)=(2,1,-1)$ or $(2,-1,1)$. When
$(p,q)=(2,1)$, the slope $\alpha$ on $K$ is
$$\mu+n(pq\mu+\lambda)=(1+pqn)\mu+n\lambda,$$
which is $1$ with respect to the frame $\lambda$, and the meridian
of the new ambient solid torus is $\mathcal M+4\mathcal L$; when
$(p,q)=(2,-1)$, the slope $\alpha$ on $K$ is $-1$, and the
meridian of the new ambient solid torus is $\mathcal M-4\mathcal
L$. We can check $\alpha$ is the blackboard frame.

If $K$ is a $1$--bridge braid, then $K$ is determined by $3$
parameters $\omega,b,t$ by Gabai \cite{G1bridge}. Here $\omega>0$
is the braid index, $1\le b\le\omega-2$, $t\equiv r \pmod{\omega}$
for some integer $r$ with $1\le r\le\omega-2$. Since the
$\alpha$--surgery yields a solid torus, by
\cite[Lemma~3.2]{G1bridge} the slope of the surgery is
$\lambda-(t\omega+d)\mu$, where $d\in\{b,b+1\}$. So
$t\omega+d=\pm1$, which is impossible for any $\omega,b,t$
satisfying the previous restrictions.
\end{proof}

\begin{lem}\label{lem:DehnTwist}
In the above lemma, let
$\varphi_{\alpha}\in\mathcal{MCG}(F,\partial F)$ be the map
induced by the $\alpha$--surgery. If $K$ is the core of $F\times
I$ and $\alpha=\frac1n$, then $\varphi_{\alpha}=\tau^n$, where
$\tau$ is the right hand Dehn twist in $F$; if $K$ is the
$(2,\pm1)$--cable in $F\times I$, then
$\varphi_{\alpha}=\tau^{\pm4}$.
\end{lem}
\begin{proof}
When $K$ is the core of $F\times I$, the conclusion is well-known.
When $K$ is the $(2,\pm1)$--cable, then from the proof of the
previous lemma we know that the meridian of the new ambient solid
torus is $\mathcal M\pm4\mathcal L$, hence the conclusion follows
from the first case.
\end{proof}

The following lemma is obvious.

\begin{lem}\label{lem:ProdDecom}
Suppose $(C\times I)\subset (F\times I)$ is a product disk or
product annulus, $(C\times I)\cap K=\emptyset$. Let $F_1$ be the
surface obtained from $F$ by cutting $F$ open along $C$, let $N_1$
be the manifold obtained from $N=(F\times
I)\backslash\mathrm{int}(\mathrm{Nd}(K))$ by cutting $N$ open
along $C\times I$. Then the pair $(N(\alpha),(\partial F)\times
I)$ is homeomorphic to $(F\times I,(\partial F)\times I)$ if and
only if the pair $(N_1(\alpha),(\partial F_1)\times I)$ is
homeomorphic to $(F_1\times I,(\partial F_1)\times I)$.\qed
\end{lem}

\begin{lem}
Theorem~\ref{thm:SurgProd} is true when $F$ is a torus.
\end{lem}
\begin{proof}
Let $C\subset F$ be a simple closed curve such that $K$ is
homologous to a multiple of $C$. Consider the homology class
$[C\times I]\in H_2(F\times I,\partial(F\times I))$, then
$[C\times I]\cdot[K]=0$. It follows that $[C\times I]$ is also a
homology class in $H_2((F\times I)\backslash K,\partial(F\times
I))$.

Let $(S,\partial S)\subset ((F\times I)\backslash K,\partial(F\times
I))$ be a taut surface representing $[C\times I]$. By
Theorem~\ref{thm:Gabai}, $S$ remains taut in at least one of the original
$F\times I$ and $N(\alpha)\cong F\times I$. Hence $S$ must be a product annulus. Cutting $F\times
I$ open along $S$, $K$ becomes a knot in $(\text{annulus}\times
I)$. Now we can apply Lemma~\ref{lem:FAnnulus} and Lemma~\ref{lem:ProdDecom} to get our
conclusion.
\end{proof}

\begin{lem}
If the conclusion of Theorem~\ref{thm:SurgProd} holds for all
knots whose exterior are $\partial(\mathrm{Nd}(K))$--atoroidal,
then the conclusion holds for all knots in $F\times I$.
\end{lem}
\begin{proof}
By the assumption, we only need to consider the case where there
is a torus in $N=F\times I\backslash\mathrm{int}(\mathrm{Nd}(K))$
which is I-cobordant but not parallel to $\partial
\mathrm{Nd}(K)$. Let $R$ be an ``innermost'' such torus.

By \cite[Lemma~3.1]{Ni2}, $R$ bounds a solid torus $U$ in $F\times
I$, such that $K\subset U$. Since $R$ is innermost in $N$, if a
torus in $(F\times I)\backslash\mathrm{int}(U)$ is I-cobordant to
$\partial U=R$, then this torus is parallel to $R$. Let $V$ be the
manifold obtained from $U$ by $\alpha$--surgery on $K$.

By Gabai \cite{GSurg}, one of the following cases holds.

\noindent1) $V=D^2\times S^1$. In this case $K$ is a $0$--bridge
or $1$--bridge braid in $U$, and the core $K'$ of the surgery
is also a $0$--bridge or $1$--bridge braid in $V$. Moreover, $K$
and $K'$ have the same braid index $\omega$.

\noindent2) $V=(D^2\times S^1)\#W$, where $W$ is a closed
$3$--manifold and $1<|H_1(W)|<\infty$.

\noindent3) $V$ is irreducible and $\partial V$ is incompressible.

Since $V\subset N(\alpha)\cong F\times I$, Cases 2) and 3) can not
happen, so the only possible case is 1). Thus the core of $U$ is a
knot such that a surgery on the knot yields the pair
$(N(\alpha),(\partial F)\times I)$ which is homeomorphic to
$(F\times I,(\partial F)\times I)$. Moreover,
$N\backslash\mathrm{int}(U)$ is $\partial U$--atoroidal. By our
assumption, the core of $U$ is a $0$--crossing or $1$--crossing
knot in $F\times I$.

If the core of $U$ is isotopic to $\eta\times\{\frac12\}$ for
some simple closed curve $\eta\subset F$, let $G\subset F$ be a
tubular neighborhood of $\eta$, then $K$ lies in $G\times I$ after
an isotopy. Let $M=(G\times
I)\backslash\mathrm{int}(\mathrm{Nd}(K))$. By
Lemma~\ref{lem:ProdDecom}, $(M(\alpha),(\partial G)\times I)$ is
homeomorphic to $(G\times I,(\partial G)\times I)$. Applying
Lemma~\ref{lem:FAnnulus}, we find that $K$ is the
$(2,\pm1)$--cable of the core of $G\times I$, and the slope
$\alpha$ is the blackboard frame $\lambda_b$.

If the core of $U$ is a $1$--crossing knot, then the blackboard
frame $\lambda'_b$ on $\partial U$ is the meridian of $V$, so
$\lambda_b'$ cobounds a punctured disk with $\omega$ oriented
copies of $\alpha$ in $U\backslash\mathrm{int}(\mathrm{Nd}(K))$.
Moreover, the meridian $\mu'$ on $\partial U$ cobounds a punctured
disk with $\omega$ oriented copies of $\mu$ in
$U\backslash\mathrm{int}(\mathrm{Nd}(K))$. Since
$[\lambda'_b]\cdot[\mu']=1$, considering the intersection of the
two punctured disks we conclude that $\omega=1$. Hence $\partial
U$ is parallel to $\partial \mathrm{Nd}(K)$, a contradiction.
\end{proof}

In light of the above lemma, from now on we assume the exterior of
the knot $K$ is $\partial(\mathrm{Nd}(K))$--atoroidal.

\section{Comparing Euler classes of foliations}\label{sec:Euler}

Let $E$ be a maximal (up to isotopy) compact essential subsurface
of $F$, such that $K$ can be isotoped in $F\times I$ to be
disjoint from $E\times I$. Let $G=\overline{F\backslash E}$.

The goal of this section is to prove the following proposition.

\begin{prop}\label{prop:AorPants}
The subsurface $G$ is either an annulus or a pair of pants.
\end{prop}

Let $T=\partial(\mathrm{Nd}(K))$, $\gamma=((\partial G)\times I)\cup
T$. Let
$$N=(F\times I)\backslash\mathrm{int}(\mathrm{Nd}(K)),M=(G\times
I)\backslash\mathrm{int}(\mathrm{Nd}(K)).$$ Then the sutured
manifold $(M,\gamma)$ contains no product disks or product annuli.
For a proper surface $S\subset M$, let $\partial_i(S)=S\cap
(G\times i)$, $i=0,1$.

Let $X=(G\times S^1)\backslash\mathrm{int}(\mathrm{Nd}(K))$ be the manifold obtained from $M$ by
gluing $G\times 1$ to $G\times 0$ via the identity map of $G$. Suppose $\xi$ is
a slope on $K$. Let $N(\xi),M(\xi),X(\xi)$ be the
manifolds obtained from $N,M,X$ by $\xi$--filling on $T$,
respectively. Let $K(\xi)\subset M(\xi)$ be the core of the new
solid torus.

By Lemma~\ref{lem:ProdDecom}, $X(\xi)$ is a surface bundle over
$S^1$ with fibre $G$ when $\xi=\infty \text{ or }\alpha$. We then
let $\mathscr E(\xi)$ be the fibration of $X(\xi)$.

\begin{lem}\label{lem:TwoEulers}
$K\subset F\times I$ is as in Theorem~\ref{thm:SurgProd}. $N$ is
$T$--atoroidal. Suppose $S\subset M$ is a taut surface such that
$S\cap T=\emptyset$ and there exists a curve $C\subset F$ with
$\partial_0S=-C\times 0,\partial_1S=C\times1$. Let $\overline S\subset X$ be the surface
obtained from $S$ by gluing $\partial_0S$ to $\partial_1S$ via the
identity map. Let $\mathscr F$ be a taut foliation of $X$ induced
by $S$. Then
$$\langle e(\mathscr F),[\overline S]\rangle=\langle e(\mathscr
E(\xi)),[\overline S]\rangle=\chi(\overline S)$$ for some $\xi\in\{\infty,\alpha\}$.
\end{lem}
\begin{proof}
By Theorem~\ref{thm:Gabai}, $S$ remains taut in $M(\xi)$ for some
$\xi\in\{\infty,\alpha\}$. Let $\mathscr F'$ be a taut foliation
of $X(\xi)$ induced by $S$. By Proposition~\ref{prop:EuEqual},
$$e(\mathscr F')=e(\mathscr E(\xi))\in H^2(X(\xi),\partial
X(\xi);\mathbb Q).$$ Since both $\mathscr F$ and $\mathscr F'$ are
induced by $S$, we have
\begin{eqnarray*}
\chi(\overline S)&=&\langle e(\mathscr F),[\overline S]\rangle\\
&=& \langle e(\mathscr
F'),[\overline S]\rangle\\
&=&\langle e(\mathscr E(\xi)),[\overline S]\rangle.
\end{eqnarray*}
\end{proof}

\begin{prop}\label{prop:SubDim1}
$K\subset F\times I$ is as in Theorem~\ref{thm:SurgProd}. $N$ is
$T$--atoroidal. The inclusion $K\subset G\times I$ induces a map
$$i_*\co H_1(K;\mathbb Q)\to H_1(G;\mathbb Q).$$
Let $$\mathcal V=\{v\in H_1(G,\partial G;\mathbb Q)|\:v\cdot i_*[K]=0\}.$$ Then the
dimension of $\mathcal V$ is at most $1$.
\end{prop}

Let
$$\rho_{\xi}\co H^2(X,\partial X;\mathbb Q)\to H^2(X(\xi),\partial
X(\xi);\mathbb Q)$$ be the map induced by the map of pairs
$$(X(\xi),\partial X(\xi))\to(X(\xi),(\partial X(\xi))\cup K(\xi)).$$

\begin{lem}\label{lem:HomoOrth}
Notations are as in Proposition~\ref{prop:SubDim1}. If the
dimension of $\mathcal V$ is greater than $1$, then there exists a
properly embedded surface $H\subset X$ such that

1) $H$ is not a multiple of $[G]$,

2) $H\cap T=\emptyset$,

3) for any two elements
$\varepsilon_{\infty}\in\rho_{\infty}^{-1}(e(\mathscr
E(\infty))),\varepsilon_{\alpha}\in\rho_{\alpha}^{-1}(e(\mathscr
E(\alpha)))$, we have $$\langle
\varepsilon_{\infty},[H]\rangle=\langle
\varepsilon_{\alpha},[H]\rangle.$$
\end{lem}
\begin{proof}
There is a natural injective map
$$\sigma\co H_1(G,\partial G)\to H_2(G\times S^1,\partial G\times S^1)$$
defined via multiplying with the $S^1$ factor. Moreover, all
elements in $\sigma(\mathcal V)$ are represented by surfaces which
are disjoint from $K$, hence $\sigma|\mathcal V$ induces an
injective map $$\widetilde{\sigma}\co\mathcal V\to H_2(X,\partial
X;\mathbb Q).$$

We pick two elements
$\varepsilon'_{\infty}\in\rho_{\infty}^{-1}(e(\mathscr
E(\infty))),\varepsilon'_{\alpha}\in\rho_{\alpha}^{-1}(e(\mathscr
E(\alpha)))$. If $\dim \mathcal V>1$, then there exists a nonzero integral
element $h\in\widetilde{\sigma}(\mathcal V)$ such
that
$$\langle \varepsilon'_{\infty},h\rangle=\langle \varepsilon'_{\alpha},h\rangle.$$
Let $H\subset X$ be a proper surface representing $h$ such that
$H\cap T=\emptyset$. We claim that this $H$ is what we need. We
only need to check 3) since the first two conditions are obvious.

From the Mayer--Vietoris sequence $$\begin{CD}
H^1(K(\xi))@>>>H^2(X,\partial X)@>\rho_{\xi}>>H^2(X(\xi),\partial
X(\xi))
\end{CD}$$
and the fact that $h\cdot[T]=0$ we conclude that $\langle
\varepsilon_{\xi},h\rangle$ does not depend on the choice of
$\varepsilon_{\xi}\in\rho^{-1}_{\xi}(e(\mathscr E(\xi)))$. Hence
3) holds.
\end{proof}

Assume the dimension of $\mathcal V$ is greater than $1$, let $H$
be a surface as in Lemma~\ref{lem:HomoOrth}, and suppose
$H\pitchfork G$. Without loss of generality, we can assume no
component of $C=H\cap G$ is nullhomologous in $H_1(G,\partial G)$,
and $H$ is efficient in $G\times S^1$, hence we can also assume
$H\cap G$ is efficient in $G$.

Let $p\in G\backslash C$ be a point. Let $\mathcal S_m(+C)$ be the
set of properly embedded oriented surfaces $S\subset G\times I$,
such that $S\cap K=\emptyset$, $\partial_0 S=-C\times0$,
$\partial_1 S=C\times1$, and the algebraic intersection number
between $S$ and $p\times I$ is $m$. Similarly, let $\mathcal
S_m(-C)$ be the set of properly embedded surfaces $S\subset
G\times I$, such that $S\cap K=\emptyset$, $\partial_0
S=C\times0$, $\partial_1 S=-C\times1$, and the algebraic
intersection number of $S$ with $p\times I$ is $m$. Since
$[C]\cdot i_{*}([K])=0$, $\mathcal S_{m}(\pm C)\ne\emptyset$.

Suppose $S\subset M$ is a properly embedded surface which is
transverse to $\partial G\times0$. For any component $S_0$ of $S$,
we define
$$y(S_0)=\max\{\frac{|S_0\cap(\partial G\times0)|}2-\chi(S_0),0\},$$
and let $y(S)$ be the sum of $y(S_i)$ with $S_i$ running over all
components of $S$. Let $y(\mathcal S_m(\pm C))$ be the minimal
value of $y(S)$ for all $S\in\mathcal S_m(\pm C)$.

\begin{lem}\label{lem:ExistTaut}
When $m$ is sufficiently large, there exist surfaces
$S_1\in\mathcal S_m(+C)$ and $S_2\in\mathcal S_m(-C)$, such that
they are taut.
\end{lem}
\begin{proof}
Let $x(\cdot)$ be the Thurston norm in $H_2(X,\partial X)$. There
exists $N\geq0$, such that if $k>N$, then
$x([H]+(k+1)[G])=x([H]+k[G])+x(G)$. As in the proof of Gabai
\cite[Theorem~3.13]{G1}, if $\overline Q$ is a Thurston norm
minimizing surface in the homology class $[H]+k[G]$, and
$\overline Q\cap G$ consists of essential curves in $G$, then $Q$
gives a taut decomposition of $M$, where $Q$ is obtained from
$\overline Q$ by cutting open along $\overline Q\cap G$. Moreover,
we can assume $\overline Q$ is efficient in $X$. Hence $\overline
Q\cap T=\emptyset$ and for each boundary component $\delta$ of
$G\times i$, $|\partial Q \cap\delta|=|[\partial
Q]\cdot[\delta]|$.

Now we can apply Gabai \cite[Lemma 0.6]{G2} to get a new taut
surface $Q'$ such that $\partial_0 Q'=-C\times 0,\partial_1
Q'=C\times 1$. When $m$ is sufficiently large, let $S_1$ be the
surface obtained by doing oriented cut-and-pastes of $Q'$ with
$(m-Q'\cdot (p\times I))$ copies of $G$, then $S_1\in\mathcal
S_m(+C)$ is the surface we need. Similarly, we can find the
surface $S_2\in\mathcal S_m(-C)$.
\end{proof}

\begin{correction}
In Ni \cite{Ni2}, after the statement of Proposition~3.4, the
author claims that there exists a circle or arc $C\subset G$ such
that $[C]\cdot i_*[K]=0$. This claim is not true. The correct
statement should be there exists an essential efficient curve $C$
in $G$ such that $[C]\cdot i_*[K]=0$. The proof only needs slight
changes: one can make use of the above Lemma~\ref{lem:ExistTaut}
to find taut surfaces.\qed
\end{correction}

The following result is Ni \cite[Lemma~3.6]{Ni2}, whose proof uses
the assumption that $(M,\gamma)$ contains no essential product
disks or product annuli and an argument from Gabai \cite{G3}.

\begin{lem}\label{lem:SumLarge}
For any positive integers $p,q$,
$$y(\mathcal S_p(+C))+y(\mathcal
S_q(-C))>(p+q)y(G).$$\qed
\end{lem}

\begin{proof}[Proof of Proposition~\ref{prop:SubDim1}]
By Lemma~\ref{lem:ExistTaut}, when $m$ is large there exist taut
surfaces $S_1\in\mathcal S_m(+C)$, $S_2\in\mathcal S_m(-C)$. By
Theorem~\ref{thm:Gabai}, $S_i$ remains taut in $M(\xi_i)$ for some
$\xi_i\in\{\infty,\alpha\}$, $i=1,2$. Let $\mathscr F_i$ be a taut
foliation of $X(\xi_i)$ induced by $S_i$.

Let $\overline{S_1},\overline{S_2}\subset X$ be the surfaces
obtained from $S_1,S_2$ by gluing $C\times 0$ to $C\times1$. We
have
$$ [\overline{S_1}]=[H]+m[G],\quad[\overline{S_2}]=-[H]+m[G]
$$
in $H_2(X,\partial X)$ and $H_2(X(\xi),\partial X(\xi))$.

We have
$$\chi(\overline{S_i})=\chi(S_i)-|\partial_0S_i|=-y(S_i),$$
and by Proposition~\ref{prop:EuEqual}
\begin{eqnarray*}
\chi(\overline{S_1})&=&\langle e(\mathscr F_1),[\overline{S_1}]\rangle\\
&=&\langle e(\mathscr E(\xi_1)),[H]+m[G]\rangle,\\
\chi(\overline{S_2})&=&\langle e(\mathscr
F_2),[\overline{S_2}]\rangle\\
&=&\langle e(\mathscr E(\xi_2)),-[H]+m[G]\rangle.
\end{eqnarray*}
By Lemma~\ref{lem:HomoOrth}, $\langle e(\mathscr
E(\xi_1)),[H]\rangle=\langle e(\mathscr E(\xi_2)),[H]\rangle$.
So
\begin{eqnarray*}
\chi(\overline{S_1})+\chi(\overline{S_2})&=&\langle
\mathscr E(\xi_1),m[G]\rangle+\langle\mathscr E(\xi_2),m[G]\rangle\\
&=&2m\chi(G),
\end{eqnarray*}
which contradicts Lemma~\ref{lem:SumLarge}.
\end{proof}

\begin{proof}[Proof of Proposition~\ref{prop:AorPants}]
By Proposition~\ref{prop:SubDim1}, $b_1(G)\le2$, so $G$ is an
annulus, a pair of pants or a genus-one surface with one boundary
component. We only need to show that the last case is not
possible.

Suppose $g(G)=1$ and $|\partial G|=1$. Let $C\subset G$ be a
simple closed curve such that $[C]\cdot i_*[K]=0$, then there exists a closed taut surface $H\subset
X$ such that $[H]=[C\times S^1]$ and
$H\cap T=\emptyset$. Since $M$ does not contain any product
annuli, $H$ is not a torus, hence $H$ is not taut in $X(\infty)$.
By Theorem~\ref{thm:Gabai}, $H$ is taut in $X(\alpha)$.

Consider the monodromy $\varphi$ of $X(\alpha)$, the surface
$H\subset X(\alpha)$ forces $\varphi_*[C]=[C]$. Since $G$ is a
once-punctured torus, $\varphi(C)$ is isotopic to $C$. Thus there
exists a torus $R\subset X(\alpha)$ such that $R\cap G=C=H\cap G$,
which implies that $[H]=[R]+m[G]$ for some integer $m$. Since $H$
is closed, $m=0$. This contradicts the facts that $H$ is taut in
$X(\alpha)$ and that $H$ is not a torus.
\end{proof}

\section{Knots in pants$\times I$}

In this section, we study the case where $G$ is a pair of pants.

The following elementary observation is stated without proof.

\begin{lem}\label{lem:EffCurve}
Suppose $C_1,C_2\subset G$ are two efficient curves consisting of
essential arcs. If they are homologous in $H_1(G,\partial G)$,
then they are isotopic.\qed
\end{lem}

Let $a,b,c$ be the three boundary components of $G$, $u,v,w$ be
three mutually disjoint oriented arcs in $G$ such that $u$
connects $b$ to $c$, $v$ connects $c$ to $a$, $w$ connects $a$ to
$b$. Then
\begin{equation}\label{eq:uvw=0}
[u]+[v]+[w]=0\in H_1(G,\partial G).
\end{equation}

\begin{lem}
None of $u,v,w$ has zero intersection number with $i_*[K]$.
\end{lem}
\begin{proof}
The argument is similar to the once-punctured torus case of
Proposition~\ref{prop:AorPants}. Assume that $[u]\cdot i_*[K]=0$,
then there exists a closed taut surface $H\subset X$ such that
$[H]=[u\times S^1]$. We may assume that $H$ is efficient in $X$,
hence $H$ has two boundary components and $H\cap T=\emptyset$. By
Lemma~\ref{lem:EffCurve}, we may assume that $H\cap G=u$.

Since $M$ does not contain any product disks, $H$ is not an
annulus, hence $H$ is not taut in $X(\infty)=G\times S^1$. By
Theorem~\ref{thm:Gabai}, $H$ is taut in $X(\alpha)$. Let $\varphi$
be the monodromy of $X(\alpha)$, then $H$ forces $\varphi(u)$ to
be homologous hence isotopic to $u$ by Lemma~\ref{lem:EffCurve}.
Thus there exists an annulus $A\subset X(\alpha)$ such that $A\cap
G=u=H\cap G$, which implies that $[H]=[A]+m[G]$ for some integers.
Since $H$ has only two boundary components, $m=0$. This
contradicts the facts that $H$ is taut in $X(\alpha)$ and $H$ is
not an annulus.
\end{proof}

\begin{lem}
The intersection number of $i_*[K]$ with each of $u,v,w$ is $\pm1$
or $\pm2$.
\end{lem}
\begin{proof}
Capping off $a$ with a disk, we get an annulus $G_a$. Now
$K\subset G_a\times I$ and the $\alpha$--surgery on $K$ does not
change the homeomorphism type of the pair $(G_a\times I,(\partial
G_a)\times I)$. By the previous lemma, $K$ is nontrivial in
$G_a\times I$. By Lemma~\ref{lem:FAnnulus}, $K$ is the core or the
$(2,\pm1)$--cable in $G_a\times I$, so $i_*[K]\cdot[u]$ is $\pm1$
or $\pm2$. The same argument applies to $v$ and $w$.
\end{proof}

Using the previous two lemmas and (\ref{eq:uvw=0}), we may assume
\begin{equation}\label{eq:uvwK}
[u]\cdot i_*[K]=[v]\cdot i_*[K]=1,\quad [w]\cdot
i_*[K]=-2,\end{equation} after reversing the orientation of $K$
and renumbering $a,b,c,u,v,w$ if necessary. We give $a,b,c$ the
boundary orientation induced from $G$, then
\begin{equation}\label{eq:uvwabc}
[v]\cdot[a]=[w]\cdot [b]=-[u]\cdot[b]=[u]\cdot[c]=1.\end{equation} See
Figure~\ref{fig:Pants} for the homology class of $K$.

\begin{figure}
\begin{picture}(340,130)
\put(50,0){\scalebox{0.50}{\includegraphics*[50pt,190pt][530pt,
450pt]{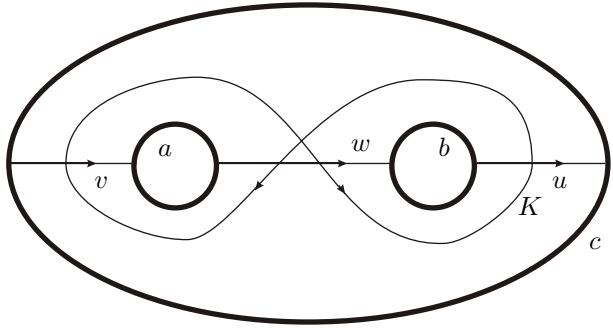}}}

\put(112,70){$a$}

\put(218,70){$b$}

\put(275,35){$c$}

\put(248,48){$K$}

\put(88,58){$v$}

\put(185,72){$w$}

\put(261,58){$u$}
\end{picture}
\caption{The homology class of $K$ in $G\times
I$}\label{fig:Pants}
\end{figure}

Let $\tau_a,\tau_b,\tau_c$ be the right-hand Dehn twists along
(parallel copies of) $a,b,c$. The mapping class group
$\mathcal{MCG}(G,\partial G)$ of $G$ is generated by
$\tau_a,\tau_b,\tau_c$. (See, for example, Farb--Margalit
\cite{FM} for preliminaries on the mapping class groups of
surfaces with boundary.) Since $a,b,c$ are disjoint,
$\mathcal{MCG}(G,\partial G)\cong \mathbb Z^3$.

\begin{lem}\label{lem:MapEle}
If $K$ is the $(2,1)$--cable in $G_c\times I$, then the map
induced by the $\alpha$--surgery is
$$\varphi_{\alpha}=\tau_a^2\tau_b^2\tau_c^{-1}.$$
If $K$ is the $(2,-1)$--cable in $G_c\times I$, then
$$\varphi_{\alpha}=\tau_a^{-2}\tau_b^{-2}\tau_c.$$
\end{lem}
\begin{proof}
Capping off $a,b$ with two disks, $G$ becomes a disk $G_{ab}$. $K$
has a canonical frame $\lambda$, which is null-homologous in
$(G_{ab}\times I)\backslash K$. Hence $\lambda$ is homologous to
$l[a]+m[b]$ in $M$ for some integers $l,m$. By (\ref{eq:uvwK}),
(\ref{eq:uvwabc}) we conclude that $\lambda$ is homologous to
$a-b$ in $M$. Hence $\lambda$ is also the canonical frame in
$G_c\times I$, where $G_c$ is obtained from $G$ by capping off $c$
with a disk.

Suppose $\varphi_{\alpha}=\tau_a^p\tau_b^q\tau_c^r$. If $K$ is the
$(2,1)$--cable in $G_c\times I$, then by Lemma~\ref{lem:FAnnulus},
the slope $\alpha$ is $1$ with respect to $\lambda$.

There is a natural map $$q_a\co \mathcal{MCG}(G,\partial G
)\to\mathcal{MCG}(G_a,\partial G_a),$$ where
$\mathcal{MCG}(G_a,\partial G_a)$ is generated by $\tau_b$. Since
$K$ is the core in $G_a\times I$ and the slope $\alpha$ is $1$,
$q_a(\varphi_{\alpha})$ must be $\tau_b$ by
Lemma~\ref{lem:DehnTwist}. The map $q_a$ sends both $\tau_b$ and
$\tau_c$ to $\tau_b$, and sends $\tau_a$ to $1$. So
$q_a(\varphi_{\alpha})=\tau_b^{q+r}$, thus $q+r=1$. The same
argument shows that $p+r=1$.

Now consider the natural map
$$q_c\co
\mathcal{MCG}(G,\partial G)\to\mathcal{MCG}(G_c,\partial
G_c)=\langle \tau_a\rangle.$$ By Lemma~\ref{lem:DehnTwist},
$q_c(\varphi_{\alpha})=\tau_a^4$. Hence $p+q=4$. So we conclude
that $p=q=2,r=-1$. The same argument works when $K$ is the
$(2,-1)$--cable in $G_c\times I$.
\end{proof}

Proposition~\ref{prop:MapEle} follows from the above lemma.

\begin{figure}
\begin{picture}(340,145)
\put(27,0){\scalebox{0.5}{\includegraphics*[10pt,220pt][580pt,
510pt]{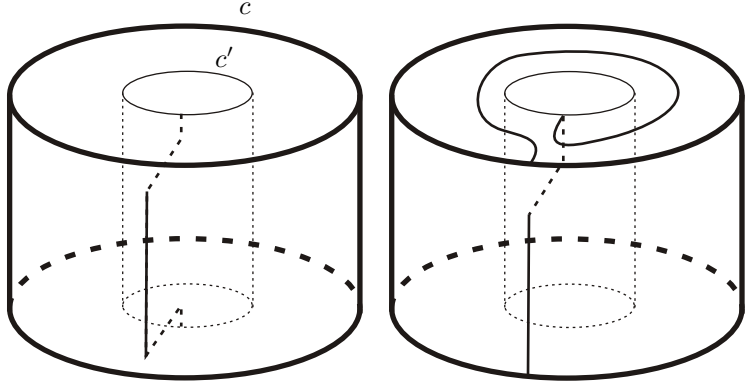}}}

\put(111,125){$c'$}

\put(120,145){$c$}

\end{picture}
\caption{Local pictures of $X(\alpha)$ near $c\times
S^1$}\label{fig:SFibre}
\end{figure}

The manifold $G\times S^1$ has a unique product structure. Let
$\omega,\omega_{\alpha}\subset c\times S^1$ be $S^1$--fibres with respect to
the product structures on $X(\infty)$ and $X(\alpha)$,
respectively.

\begin{lem}\label{lem:SFibre}
If $K$ is the $(2,1)$--cable in $G_c\times I$, then
$$[\omega_{\alpha}]=[\omega]+[c].$$
\end{lem}
\begin{proof}
The manifold $X(\infty)$ is obtained from $G\times I$ by
identifying $(x,0)$ with $(x,1)$ for each $x\in G$. By
Lemma~\ref{lem:Monodromy}, $X(\alpha)$ is obtained from $G\times
I$ by identifying $(x,0)$ with $(\varphi_{\alpha}(x),1)$ for each
$x\in G$. Choose parallel copies of $a,b,c$ in $G$, denoted
$a',b',c'$. Let $\varphi_{\alpha}$ be supported in the three
annuli bounded by $a-a'$, $b-b'$ and $c-c'$. Pick points $p\in c,
p'\in c'$, then $p'\times S^1$ is an $S^1$--fibre of the product
structures on both $X(\infty)$ and $X(\alpha)$, while $p\times
S^1$ is an $S^1$--fibre of the product structure on $X(\infty)$.

In $X(\alpha)$, we isotope $p'\times S^1$ such that it becomes a
curve $\mathcal S$ which is the union of four segments
$J,J_{\epsilon},J_{1-\epsilon},J'$, where $J$ is a vertical
segment in the interior of $c\times I$, $J_{\epsilon}\subset
G\times\epsilon, J_{1-\epsilon}\subset G\times(1-\epsilon)$, $J'$
is a vertical segment in $c'\times S^1$. See the left hand side of
Figure~\ref{fig:SFibre}.

As on the right hand side of Figure~\ref{fig:SFibre}, we push the
previous curve $\mathcal S$ down in distance $\epsilon$ to get a
new curve $\mathcal S_-$, then $J_{\epsilon}$ becomes an arc on
$G\times1$. Using Lemma~\ref{lem:Monodromy}, this new arc is
$\varphi_{\alpha}(J_{\epsilon})=\tau^{-1}_c(J_{\epsilon})$.
$\mathcal S_-$ is a fibre of $X(\alpha)$, and it is homologous to
$[p\times S^1]+[c\times 1]$. Hence our conclusion holds.
\end{proof}

\begin{lem}\label{lem:DSCutPaste}
Let $C=v-u$. Pick a point $p\in c\backslash(\partial C)$, we can
then define $\mathcal S_m(\pm C)$ as in Section~\ref{sec:Euler}.
Then there exists a connected surface $S\in S_{1}(C)$ such that
$y(S)=1$. Moreover, Let $S'\subset G\times[0,1]$ be the surface
obtained from $-C\times I$ and $G\times0$ by oriented cut-and-pastes, then $S$ is isotopic to $S'$ in $G\times[0,1]$.
\end{lem}
\begin{proof}
For any homology class $h\in H_2(X,(\partial G)\times S^1)$, let
$x(h),x_{\infty}(h),x_{\alpha}(h)$ denote its Thurston norm in
$X,X(\infty),X(\alpha)$, respectively.

Let $U=-[u\times S^1], V=-[v\times S^1]\in H_2(G\times
S^1,(\partial G)\times S^1)$. Since $(V-U)\cdot[K]=0$, $V-U$ also
represents an element in $H_2(X,(\partial G)\times S^1)$. Note
that the Thurston norm of $h\in H_2(G\times S^1,(\partial G)\times
S^1)$ is the absolute value of its algebraic intersection number with the
$S^1$--fibre. Consider $V-U+m[G]$ for $m\ge0$, using
Lemma~\ref{lem:SFibre}, we can compute
$$x_{\infty}(V-U+m[G])=m,$$
$$x_{\alpha}(V-U+m[G])=(V-U+m[G])\cdot([\omega]+[c])=|m-2|.$$

Since $x_{\infty}(V-U+[G])=x_{\alpha}(V-U+[G])=1$,
Theorem~\ref{thm:Gabai} implies that $x(V-U+[G])=1$. Let
$\overline{S}\subset X$ be a taut surface in this homology class
such that $\overline{S}$ is efficient in $X$. Then $\overline{S}$
is disjoint from $T$. Isotope $\overline{S}$ so that it is
transverse to $G$ and its intersection with $G$ contains no
trivial loops. Now $\overline{S}\cap G$ is homologous to $C$.
Moreover, $\overline S\cap G$ can be made efficient in $G$. So
$\overline{S}\cap G$ is isotopic to $C$ by
Lemma~\ref{lem:EffCurve}. Without loss of generality, we can
assume
$$\overline{S}\cap G=C\quad\text{ and}
\quad\overline{S}\cap((\partial C)\times S^1)\subset G.$$

Cutting $\overline{S}$ open along $C$, we get a surface
$S\in\mathcal S_1(+C)$ such that $y(S)=1$. After an isotopy of $S$, we can assume the two surfaces $S,C\times[0,1]\subset G\times[0,1]$ are transverse.
Since $S\cap((\partial C)\times(0,1))=\emptyset$, $S\cap(C\times(0,1))$ consists of closed curves which bounds disks in $C\times(0,1)$.
Since $S$ is incompressible and $G\times[0,1]$ is irreducible, we can isotope $S$ such that $S\cap(C\times(0,1))=\emptyset$, hence $S\cap(C\times[0,1])=C\times\{0,1\}$.
Now we glue $S$ and $C\times[0,1]$ together along $C\times\{0,1\}$ and perturb the resulting surface slightly, then we get a connected surface $G'$ with $x(G')=1$ and $\partial G'$ is parallel to
$(\partial G)\times0$ in $(\partial G)\times[0,1]$. Hence $G'$ is parallel to $G\times0$ in $G\times[0,1]$. It follows that $S$ is isotopic to $S'$ in $G\times[0,1]$.
\end{proof}

\begin{lem}\label{lem:KOnR}
Let $S$ be the surface obtained in Lemma~\ref{lem:DSCutPaste}. Let $$G\times
I\stackrel{S}{\rightsquigarrow}(M_1(\infty),\gamma_1)$$ be the
sutured manifold decomposition associated with $S$, then
$(M_1(\infty),\gamma_1)$ is a product manifold, and there is an
ambient isotopy of $M_1(\infty)$ which takes $K$ to a curve in
$R_+(\gamma_1)$ such that the frame of $K$ specified by
$R_+(\gamma_1)$ is $\alpha$.
\end{lem}
\begin{proof}
By Lemma~\ref{lem:DSCutPaste}, $S$ is obtained from
$-C\times I$ and $G\times0$ by oriented cut-and-pastes. So $(M_1(\infty),\gamma_1)$ is a product sutured manifold and
$R_+(\gamma_1)$ is an annulus.

Let $(M_1(\alpha),\gamma_1)$ be the sutured manifold obtained from
$M(\alpha)$ by decomposing along $S$. Then $M_1(\alpha)$ can also
be obtained from $M_1(\infty)$ by $\alpha$--surgery on $K$.

We claim that $M_1(\alpha)$ is not taut. In fact, let $S''$ be the
surface obtained from $S$ and $G\times0$ by oriented
cut-and-pastes. Let $\overline{S''}\subset X$ be the surface
obtained from $S''$ by gluing $\partial_0S''$ to $\partial_1S''$.
Then $x(\overline{S''})=2$ and $\overline{S''}$ represents
$V-U+2[G]$. We already computed
$$x_{\infty}(V-U+2[G])=2>x_{\alpha}(V-U+2[G])=0,$$ so $\overline{S''}$
is not taut in $X(\alpha)$. Let $M''(\alpha)$ be the non-taut
sutured manifold obtained by decomposing $X(\alpha)$ along
$\overline{S''}$.

Since $S''$ is obtained from $S$ and $G\times0$ by oriented
cut-and-pastes, and $S\cap(G\times 0)=-C\times0$ consists of two
arcs, there exist two product disks in $M''(\alpha)$ such that the
result of decomposing $M''(\alpha)$ along these two disks is
$(M_1(\alpha),\gamma_1)$. See the proof of Gabai
\cite[Theorem~3.13]{G1} for an explanation of this fact. So
$(M_1(\alpha),\gamma_1)$ is not taut by Gabai
\cite[Lemma~0.4]{G2}.

Now Theorem~\ref{thm:NormRed} implies our conclusion.
\end{proof}

\begin{figure}
\begin{picture}(340,212)
\put(72,0){\scalebox{0.4}{\includegraphics*[40pt,20pt][530pt,
550pt]{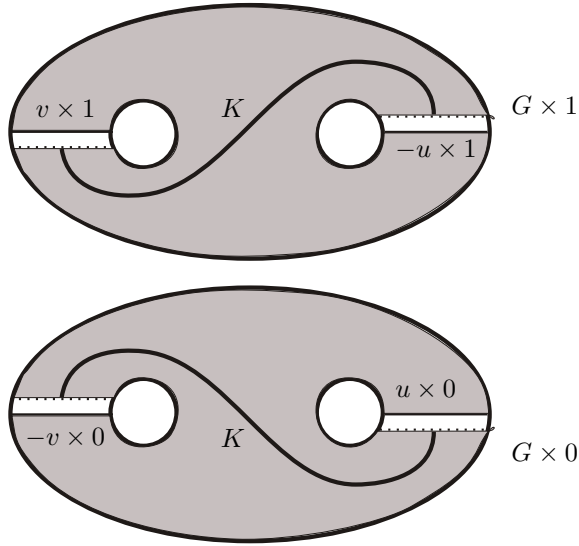}}}

\put(270,35){$G\times0$}

\put(270,166){$G\times1$}

\put(90,165){$v\times1$}

\put(86,41){$-v\times0$}

\put(226,150){$-u\times1$}

\put(226,59){$u\times0$}

\put(160,165){$K$}

\put(160,40){$K$}
\end{picture}
\caption{The surface $R_+(\gamma_1)$ containing the knot $K$
}\label{fig:DSurface}
\end{figure}

\begin{proof}[Proof of Theorem~\ref{thm:SurgProd}]
By the results in Sections~\ref{sect:WarmUp} and \ref{sec:Euler},
we only need to consider the case $F=G$ is a pair of pants. By
Lemma~\ref{lem:KOnR}, $K$ lies on $R_+(\gamma_1)$, and the frame
specified by $R_+(\gamma_1)$ is $\alpha$.

Since $R_+(\gamma_1)$ is an annulus, the only essential curve on
it is its core. As in Figure~\ref{fig:DSurface}, $R_+(\gamma_1)$
can be constructed in the following way. Cut $G\times\{0,1\}$ open
along $(v-u)\times\{0,1\}$, we get two octagons $P_0,P_1$. There
are two edges of $P_0$ which are copies of $v\times0$ with
different orientations. We call these two edges
$v\times0,-v\times0$. Similarly, there are edges $\pm u\times0,\pm
v\times1,\pm u\times1$. Now we glue two product disks to
$P_0,P_1$, such that one product disk connects $v\times0$ to
$-v\times1$ and the other connects $-u\times0$ to $u\times1$. The
annulus we get is isotopic to $R_+(\gamma_1)$. The core of this
annulus is clearly a one-crossing knot in $G\times I$. The result
about the frame also follows since the vertical projection $p\co
R_+(\gamma_1)\to G$ is an immersion.
\end{proof}

\end{document}